\documentclass[12pt]{article}
\usepackage[english]{babel}
\usepackage{amsmath,amsthm}
\usepackage{amsfonts}
\usepackage{graphicx}
\usepackage[english]{babel}
\usepackage{comment}
\usepackage{bbm}
\newtheorem{thm}{Theorem}[section]
\newtheorem{cor}[thm]{Corollary}
\newtheorem{lem}[thm]{Lemma}

\theoremstyle{definition}

\theoremstyle{remark}
\newtheorem{rem}[thm]{Remark}
\numberwithin{equation}{section}

 \newcommand\R{\mathbb{R}}

 \newcommand\E{\mathbb{E}}

 \DeclareMathOperator*{\var}{Var}

\begin{document}

\title{On the Besov regularity of the bifractional Brownian motion}%
\author{Brahim Boufoussi and Yassine Nachit\\
Department of Mathematics, Faculty of Sciences Semlalia,\\
Cadi Ayyad University, 2390 Marrakesh, Morocco\\
boufoussi@uca.ac.ma, yassine.nachit.fssm@gmail.com}%
 \date{}
\maketitle
\begin{abstract}
  Our aim in this paper is to improve H\"{o}lder continuity results for the bifractional Brownian motion (bBm) $(B^{\alpha,\beta}(t))_{t\in[0,1]
}$ with $0<\alpha<1$ and $0<\beta\leq 1$. We prove that almost all paths of the bBm belong (resp. do not belong) to  the Besov spaces $\mathbf{Bes}(\alpha \beta,p)$ (resp. $\mathbf{bes}(\alpha \beta,p)$) for any $\frac{1}{\alpha \beta}<p<\infty$, where $\mathbf{bes}(\alpha \beta,p)$ is a separable subspace of $\mathbf{Bes}(\alpha \beta,p)$. We also show the It\^{o}-Nisio theorem for the bBm with $\alpha \beta>\frac{1}{2}$ in the H\"{o}lder spaces $\mathcal{C}^{\gamma}$, with $\gamma<\alpha \beta$.
\end{abstract}



\section{Introduction}
Let $(B^{\alpha,\beta}(t))_{t\geq 0}$ be a bifractional Brownian motion (bBm for short), i.e., a centred real-valued Gaussian process with covariance function
\begin{equation}\label{R}
    R^{\alpha,\beta }(s,t):=R(s,t)=\frac{1}{2^\beta }\left((t^{2\alpha}+s^{2\alpha})^\beta -|t-s|^{2\alpha \beta }\right),
  \end{equation}
 where $\alpha\in (0,1)$ and $\beta \in (0,1]$. Point out that, when $\beta =1$,  $B^{\alpha,1}$ is a fractional Brownian motion with Hurst parameter $\alpha \in (0,1)$. However the increments of $B^{\alpha ,\beta }$ are not stationary except for the case when $\beta =1$. The bBm has the following general properties: it is self-similar with index $\alpha \beta $, that is, for every $a>0$,
  \begin{equation}\label{self similar}
    \{B^{\alpha ,\beta }(at),\;t\geq 0\}\overset{d}{=}\{a^{\alpha \beta }B^{\alpha ,\beta }(t),\;t\geq 0\},
  \end{equation}
  where $X\overset{d}{=}Y$ means that the two processes have the same finite-dimensional distributions. It is a quasi-helix
(see \cite{Kahane81} and \cite{Kahane85} for various properties and applications of quasi-helices) since for every $s,t\in [0,T]$, we have
\begin{equation}\label{var increment}
    2^{-\beta }|t-s|^{2\alpha \beta }\leq\E\left(B^{\alpha ,\beta }(t)-B^{\alpha ,\beta }(s)\right)^2\leq 2^{1-\beta }|t-s|^{2\alpha \beta }.
  \end{equation}
Based on the fractional Brownian motion structure, Houdr\'e and Villa \cite{Villa} have constructed the bifractional Brownian motion as a more general self-similar Gaussian process. Russo and Tudor \cite{RussoTudor} have shown that the bBm behaves like a fractional Brownian motion with Hurst parameter $\alpha \beta$.  There is a rich literature investigating the properties of the bifractional Brownian motion, we refer for example to the following non-exhaustive list: Bojdecki et al. \cite{BojdeckiGorostizaTalarczyk}, El-Nouty \cite{Nouty}, El-Nouty and Journ\'e \cite{NoutyJourne}, Kruk et al. \cite{KrukRussoTudor}, Es-Sebaiy and Tudor \cite{SebaiyTudor}, Tudor and Xiao \cite{TudorXiao} and Lei and Nualart \cite{LeiNualart}, just to mention a few. It was shown essentially in this last paper the following decomposition of the bBm
$$\{C_2B^{\alpha \beta }(t),\;t\geq 0\}\overset{d}{=} \{C_1X^{\alpha ,\beta }(t)+B^{\alpha ,\beta }(t),\; t\geq 0\},$$
where $C_1$, $C_2$ are two constants and $(B^{\alpha \beta }(t))_{t\geq 0}$ is a fractional Brownian motion (fBm) with parameter $\alpha \beta $ and $(X^{\alpha ,\beta }(t))_{t\geq 0}$ is a Gaussian process with infinitely differentiable trajectories on $(0, +\infty)$ and absolutely continuous on $[0, +\infty)$. On the other hand, we know from Ciesielski et al. \cite{Roynette} that almost all paths of the fBm $(B^{\alpha \beta }(t))_{t\geq 0}$ belong (resp. do not belong) to the Besov space $\mathbf{Bes}(\alpha \beta ,p)$
(resp. to the separable subspace $\mathbf{bes}(\alpha \beta,p)$) (definitions are given in Section 2). Needless to mention that, if we take $0<a<b$ one can deduce directly  by Lei and Nualart decomposition that the sample paths of $(B^{\alpha ,\beta }(t))_{a\leq t\leq b}$ satisfy the same Besov regularity as  those of fractional Brownian motion of parameter $\alpha \beta$. Otherwise, we are unable to get the H\"{o}lder regularity of  $X^{\alpha ,\beta }$ on intervals of type $[0, \varepsilon] $, for $\varepsilon>0$, since the trajectories of this process are only  absolutely continuous near $0$. Hence we can not derive directly from Lei and Nualart decomposition the Besov regularity for the bBm on the interval $[0,\varepsilon]$. Our main purpose in this paper is to investigate the Besov regularity for sample paths of  the bBm
$(B^{\alpha ,\beta }(t)) $ for $t\in[0, 1]$.

Besov spaces $ \mathbf{Bes}(\gamma, p)$ are a general framework to investigate the modulus of smoothness  in $L^p$-norms for trajectories of continuous time stochastic processes. In our paper we are concerned by a particular class of Besov spaces of real functions $(f(t)\,,\,\,\, t\in [0,1])$ (for a more general context we can see Triebel \cite{Triebel}).
We note by $\mathcal{C}^{\gamma}$ the space of functions satisfying a H\"{o}lder condition of order $\gamma>0$ endowed with the usual norm. It's known that
Besov spaces cover the H\"{o}lder spaces as particular cases, more precisely $\mathcal{C}^{\gamma}=\mathbf{Bes}(\gamma,\infty)$. And, for $p$ large enough,
we have the following
continuous injections (see Section 2):
$$\mathcal{C}^{\gamma +\frac{1}{p}}\hookrightarrow
\mathbf{bes}(\gamma , p) \hookrightarrow \mathbf{Bes}(\gamma, p) \hookrightarrow \mathcal{C}^{\gamma -\frac{1}{p}}, $$
where $ \mathbf{bes}(\gamma, p)  $ is a separable subspace of the Besov space
$\mathbf{Bes}(\gamma, p)$.
It is well known that almost surly the sample paths of the bBm $(B^{\alpha, \beta }(t))_{t\geq 0}$ belong to the H\"{o}lder spaces $\mathcal{C}^{\gamma}$ for $\gamma < \alpha \beta $, and do not belong a.s. to $\mathcal{C}^{\alpha \beta }$. Our aim is to improve these classical results by showing that we can get smoothness of order $\alpha \beta $ in the Besov spaces $\mathbf{Bes}(\alpha \beta ,p)$ for $\frac{1}{\alpha \beta }<p<\infty $. This is the best regularity one can get in the context of Besov spaces, because we  also prove that almost surly the trajectories of the bBm do not belong to the separable spaces $\mathbf{bes}(\alpha \beta, p)$ for $\frac{1}{\alpha \beta }<p<\infty$. So
the above injections explain clearly the sharpness of our results. Note that our paper leads to some previous Besov regularity results: For $\beta=1 $ we recover the fBm situation considered in \cite{Roynette}; and the case $\alpha=\beta=\frac{1}{2} $ corresponds to the regularity of the mild solution for a linear stochastic heat equation driven by a white noise (see \cite{bounachit}).

Among It\^{o}'s accomplishments, there is the  It\^{o}-Nisio theorem (cf. \cite{Ito}), in which the authors have established on one hand a general improvement of the Fourier series decomposition of the Brownian motion, and on the other hand  a generalization of Wiener's construction of the Brownian motion. They have given the expansion as the convergence of normalized sums of independent random variables. Later, Kerkyacharian and Roynette \cite{KerkyacharianRoynette} have proved the same result of the It\^{o}-Nisio in H\"{o}lder spaces with a sample proof. In this paper we show the It\^{o}-Nisio theorem for the bBm with $\alpha \beta >\frac{1}{2}$ in the H\"{o}lder spaces $\mathcal{C}^{\gamma}$, with $\gamma<\alpha \beta$. The case $\beta=1$ corresponds to the fBm with Hurst parameter $\alpha>1/2$.


This paper is organized as follows. In the second paragraph we give a brief introduction to  Besov spaces. The third paragraph is devoted to study the Besov regularity for the sample paths of the bifractional Brownian motion. In the fourth paragraph we investigate the It\^{o}-Nisio theorem for bBm with $\alpha \beta >\frac{1}{2}.$
The proofs of our results use technical and very fine calculations based on dyadic coordinate expansions of the
bifractional Brownian motion and descriptions of the Besov norms in terms of the corresponding
expansion coefficients of a function.
\section{Preliminaries}
\subsection{Besov spaces}
Let $I\subset\R$ be a compact interval , $1\leq p<\infty$ and $f\in L^p(I\,;\,\R)$. We define for any $t>0$
$$\Delta_{p}(f,I)(t)=\sup_{|s|\leq t}\left\{\int_{I_s}|f(x+s)-f(x)|^pdx\right\}^{\frac{1}{p}},$$
where $I_s=\{x\in I;\;\;x+s\in I\}$. $\Delta_{p}(f,I)(t)$ denotes the modulus of continuity of $f$ in the $L^p$-norm. For $\gamma>0$, we consider the norm
$$||f||_{\gamma,p}:=||f||_{L^p(I)}+\sup_{0< t\leq 1}\frac{\Delta_{p}(f,I)(t)}{t^{\gamma}}.$$
The Besov space is given by
$\mathbf{Bes}(\gamma,p)(I)=\{f\in L^p(I);\;\; ||f||_{\gamma,p}<\infty\}.$
The space $(\mathbf{Bes}(\gamma,p)(I), ||.||_{\gamma,p})$ is a non separable Banach space. We also define
$\mathbf{bes}(\gamma, p)(I)=\{f\in L^p(I);\;\; \Delta_{p}(f, I)(t)=o(t^{\gamma})\;\text{as}\;t\rightarrow 0^+\} $
a separable subspace of $\mathbf{Bes}(\gamma,p)(I)$.
For $p=\infty$, the space $\mathbf{Bes}(\gamma,\infty)(I) $ is defined in the same way by using the usual $L^\infty$-norm.

In the case of unit interval $I=[0,1] $ Besov spaces are characterized in terms of sequences of the coefficients of the expansion of continuous functions with respect to the Schauder basis. The following isomorphism theorem has been established by Ciesielski et al. \cite{Roynette}
\begin{thm}\label{Besov}
Let $1<p<\infty$ and $\frac{1}{p}<\gamma<1$, we have
\begin{enumerate}
  \item  $\mathbf{Bes}(\gamma,p)([0,1])$ is linearly isomorphic to a sequences space and we have the following equivalence of norms:
      $$||f||_{\gamma,p}\sim\sup\left\{|f_0|,|f_1|,\sup_{j}2^{-j\left(\tfrac{1}{2}-\gamma+\tfrac{1}{p}\right)}\left[\sum_{k=1}^{2^{j}}|f_{jk}|^p\right]^{\tfrac{1}{p}} \right\},$$
where the coefficients $\left\{f_0,  f_1, f_{jk}\,, j\geq0\,, 1\leq k\leq 2^{j}\right\}$ are given by
$$f_0=f(0),\qquad f_1=f(1)-f(0),$$
$$f_{jk}=2\cdot 2^{j/2}\left\{f\left(\frac{2k-1}{2^{j+1}}\right)-\frac{1}{2}f\left(\frac{2k}{2^{j+1}}\right)-\frac{1}{2}f\left(\frac{2k-2}{2^{j+1}}\right)\right\}.$$
  \item $f$ is in $\mathbf{bes}(\gamma, p)([0,1])$ if and only if
  $$\lim_{j\to 0}2^{-j\left(\tfrac{1}{2}-\gamma+\tfrac{1}{p}\right)}\left[\sum_{k=1}^{2^j}|f_{jk}|^p\right]^{\tfrac{1}{p}}=0.$$
\end{enumerate}
\end{thm}
\begin{rem}
  \begin{enumerate}
    \item Let $1\leq p<\infty$ and $0<\gamma<\gamma'<1$, then we have
    $$\mathbf{Bes}(\gamma',p)(I)\hookrightarrow \mathbf{bes}(\gamma, p)(I).$$
    \item We denote by $\mathcal{C}^{\gamma}(I)$ the H\"{o}lder space define by
    \begin{equation}\label{C omega}
      \mathcal{C}^{\gamma}(I):=\left\{f\in C(I),\;\; \sup_{x,y\in I \atop x\neq y }\frac{|f(x)-f(y)|}{|x-y|^{\gamma}}<\infty\right\},
    \end{equation}
    endowed with the norm
    $||f||_{\gamma}=\sup_{x\in I}|f(x)|+\sup_{x,y\in I \atop x\neq y }\frac{|f(x)-f(y)|}{|x-y|^{\gamma}}.$
    \begin{itemize}
      \item $\mathbf{Bes}(\gamma,\infty)(I)=\mathcal{C}^{\gamma}(I).$
      \item For $ 1\leq p<\infty$, we have $\mathcal{C}^{\gamma}(I)\hookrightarrow \mathbf{Bes}(\gamma,p)(I).$
      \item For $ \frac{1}{p}<\gamma<1$, we obtain $$\mathcal{C}^{\gamma+\frac{1}{p}}(I)\hookrightarrow \mathbf{bes}(\gamma, p)(I) \hookrightarrow \mathbf{Bes}(\gamma,p)(I) \hookrightarrow \mathcal{C}^{\gamma-\frac{1}{p}}(I).$$
    \end{itemize}
  \end{enumerate}
\end{rem}
In the next, we will
restrict ourselvse to the interval $I=[0,1]$, so we will omit to precise the interval $I$ in our notations, e.g.
$\mathbf{Bes}(\gamma,p):=\mathbf{Bes}(\gamma,p)(I) $.
\section{Besov regularity of the bifractional Brownian motion}
Our main result is the following theorem
\begin{thm}\label{pricipal BHK}
For each $\alpha \in (0,1)$, $\beta\in (0,1]$ and $\frac{1}{\alpha \beta}<p<\infty$, we have
$$\mathbb{P}(B^{\alpha ,\beta}(.)\in \mathbf{Bes}(\alpha \beta,p) )=1\;\;\text{and}\;\;\mathbb{P}(B^{\alpha ,\beta}(.)\in \mathbf{bes}(\alpha \beta, p) )=0,$$
where $B^{\alpha ,\beta}(.)$ are the sample paths $ t\in [0,1]\to B^{\alpha ,\beta}(t) $.
\end{thm}
\noindent To show this theorem, we will adapt the techniques in \cite{Roynette}. Let us first give some preliminary results.

The below lemma is a useful tool to obtain precise estimations in the calculations of this paper. For the proof we refer to \cite{Roynette}.
\begin{lem}\label{cauchy gaussian}
  Let $(X,Y)$ be a mean zero Gaussian vector such that $\E(X^2)=\E(Y^2)=1$ and $\rho=|\E XY|$. So for any measurable functions $f$ and $g$ such that
  $\E(f(X))^2<\infty,\E(f(Y))^2<\infty$ and $f(X)$, $f(Y)$ are centred, we have
  $$|\E f(X)g(Y)|\leq \rho \left\{\E(f(X))^2\right\}^{1/2}\left\{\E(f(Y))^2\right\}^{1/2},$$
  when $f$ (or $g$) is even we can replace $\rho$ by $\rho^2$ in the previous  inequality.
\end{lem}
We define
\begin{equation}\label{ujk}
  u_{jk}:=2\cdot 2^{j/2}\left\{B^{\alpha ,\beta}\left(\frac{2k-1}{2^{j+1}}\right)-\frac{1}{2}B^{\alpha ,\beta}\left(\frac{2k}{2^{j+1}}\right)-\frac{1}{2}B^{\alpha ,\beta}\left(\frac{2k-2}{2^{j+1}}\right)\right\}.
\end{equation}
We set
\begin{equation}\label{vjk}
  v_{jk}=\frac{u_{jk}}{\sigma_{jk}}\qquad\text{with}\qquad \sigma_{jk}=\left\{\E[|u_{jk}|^2]\right\}^{1/2}.
\end{equation}
By using \eqref{R} and \eqref{ujk}, we have for all $j\geq 1$ and $k,k'\in \{1,...,2^j\}$
\begin{equation}\label{ujkujkprime}
  \E[u_{jk}u_{jk'}]=\frac{2^{j(1-2\alpha \beta)}}{2^{\beta+2\alpha \beta}}(\Delta^2_y\Delta^2_x\Psi_{k,k'}(0,0)-\Delta^4\Phi_{k,k'}(0)),
\end{equation}
where $\Delta^2_y$ (resp. $\Delta^2_x$) is the one step progressive difference of order $2$ in the $y$ variable (resp. $x$ variale), and   $\Delta^4$  is the one step progressive difference of order $4$. The functions in \eqref{ujkujkprime} are
$$\Psi_{k,k'}(x,y)=\left((2k-2+x)^{2\alpha }+(2k'-2+y)^{2\alpha }\right)^\beta,$$
and
$$\Phi_{k,k'}(x)=\left|2(k-k')-2+x\right|^{2\alpha \beta}.$$
\begin{lem}\label{lem estimation Eujkujkprime}
  For all $\alpha \in (0,1)$, $\beta\in (0,1]$, $j\geq 1$ and $k,k'\in \{1,...,2^j\}$ with $k'<k$, there exist $C>0$, $\kappa_{k,k'}\in (0,2)$ and $c_{k,k'}\in (0,4)$ such that
  \begin{equation}\label{estmation Eujkujkprime}
        \begin{split}
           \left|\E[u_{jk}u_{jk'}]\right|\leq C 2^{j(1-2\alpha \beta)}&\left\{\frac{1}{(2k'-2+\kappa_{k,k'})^{4-2\alpha \beta}}\right.\\
             & \left.+\frac{1}{(2(k-k')-2+c_{k,k'})^{4-2\alpha \beta}}\right\}.
        \end{split}
  \end{equation}
  And there exist two constants $m_1,m_2>0$ such that, for all $j\geq 1$ and $k\in \{1,...,2^j\},$
  \begin{equation}\label{estmation Eujk2}
    m_1\, 2^{j(1-2\alpha \beta)} \leq\E[|u_{jk}|^2] \leq m_2\, 2^{j(1-2\alpha \beta)}.
  \end{equation}
\end{lem}
\begin{proof}
  Denote by $\Phi_{k,k'}^{(4)}$ the derivative of order $4$ of $\Phi_{k,k'}$. So by the mean value theorem and \eqref{ujkujkprime}, there exist three  constants $c_{1,k,k'},c_{2,k,k'}\in (0,2)$ and $c_{3,k,k'}\in (0,4)$ such that
  \begin{equation}\label{derivative}
    \E[u_{jk}u_{jk'}]=\frac{2^{j(1-2\alpha \beta)}}{2^{\beta+2\alpha \beta}}(\partial^2_y\partial^2_x\Psi_{k,k'}(c_{1,k,k'},c_{2,k,k'})-\Phi_{k,k'}^{(4)}(c_{3,k,k'})),
  \end{equation}
  where $\partial_y=\frac{\partial}{\partial y}$ and $\partial_x=\frac{\partial}{\partial x}.$ On the other hand, we get for all $j\geq 1$ and $k,k'\in \{1,...,2^j\},$
  \begin{equation}\label{derivative of Psi}
    \begin{split}
     & \partial^2_y\partial^2_x\Psi_{k,k'}(c_{1,k,k'},c_{2,k,k'}) \\
     &= 4\alpha ^2(2\alpha -1)^2\beta(\beta-1)(2k-2+c_{1,k,k'})^{2\alpha -2}(2k'-2+c_{2,k,k'})^{2\alpha -2}\\
     &\qquad\qquad\qquad\qquad\times\left((2k-2+c_{1,k,k'})^{2\alpha }+(2k'-2+c_{2,k,k'})^{2\alpha }\right)^{\beta-2} \\
     &+ 8\alpha ^3(2\alpha -1)\beta(\beta-1)(\beta-2)(2k-2+c_{1,k,k'})^{2\alpha -2}(2k'-2+c_{2,k,k'})^{4\alpha -2}\\
     &\qquad\qquad\qquad\qquad\times\left((2k-2+c_{1,k,k'})^{2\alpha }+(2k'-2+c_{2,k,k'})^{2\alpha }\right)^{\beta-3} \\
     &+ 8\alpha ^3(2\alpha -1)\beta(\beta-1)(\beta-2)(2k-2+c_{1,k,k'})^{4\alpha -2}(2k'-2+c_{2,k,k'})^{2\alpha -2}\\
     &\qquad\qquad\qquad\qquad\times\left((2k-2+c_{1,k,k'})^{2\alpha }+(2k'-2+c_{2,k,k'})^{2\alpha }\right)^{\beta-3}\\
     &+ 16\alpha ^4\beta(\beta-1)(\beta-2)(\beta-3)(2k-2+c_{1,k,k'})^{4\alpha -2}(2k'-2+c_{2,k,k'})^{4\alpha -2}\\
     &\qquad\qquad\qquad\qquad\times\left((2k-2+c_{1,k,k'})^{2\alpha }+(2k'-2+c_{2,k,k'})^{2\alpha }\right)^{\beta-4}\\
     &=I_1+I_2+I_3+I_4.
     \end{split}
  \end{equation}
  And for all $j\geq 1$ and $k,k'\in \{1,...,2^j\}$ such that $k'<k$, we have
  \begin{equation}\label{derivative of Phi}
    \Phi_{k,k'}^{(4)}(c_{3,k,k'})=\prod_{l=0}^{3}(2\alpha \beta-l)(2(k-k')-2+c_{3,k,k'})^{2\alpha \beta-4}.
  \end{equation}
  Let us first investigate the inequality \eqref{estmation Eujkujkprime}. Combining \eqref{derivative}, \eqref{derivative of Psi} and \eqref{derivative of Phi}, we get for all $j\geq 1$ and $k,k'\in \{1,...,2^j\}$ such that $k'<k$
  \begin{equation}\label{estimation ujkujkprime}
    \left|\E[u_{jk}u_{jk'}]\right|\leq \frac{2^{j(1-2\alpha \beta)}}{2^{\beta+2\alpha \beta}}\left\{\sum_{l=1}^{4}|I_l|+\frac{\prod_{l=0}^{3}|2\alpha \beta-l|}{(2(k-k')-2+c_{3,k,k'})^{4-2\alpha \beta}}\right\}.
  \end{equation}
  \begin{itemize}
    \item For $\alpha \leq \frac{1}{2}$. Let $j\geq 1$ and $k,k'\in \{1,...,2^j\}$ such that $k'<k$, we have by \eqref{derivative of Psi}
        \begin{equation}\label{Il}
          \sum_{l=1}^{4}|I_l|\leq \frac{ \tilde{C}(\alpha ,\beta)}{(2k'-2+c_{1,k,k'}\wedge c_{2,k,k'})^{4-2\alpha \beta}},
        \end{equation}
        where
        \begin{equation}\label{C tilde}
          \begin{split}
              \tilde{C}(\alpha ,\beta)=2^\beta\alpha ^2|\beta-1|\{&\beta(2\alpha -1)^2+2\alpha |2\alpha -1|\beta|\beta-2|\\
               & +\alpha ^2\beta|\beta-2||\beta-3|\}.
          \end{split}
        \end{equation}
        Combining \eqref{estimation ujkujkprime} and \eqref{Il}, we get
        \begin{equation}\label{estimation ujkujkprime H ineferiere 12}
        \begin{split}
           \left|\E[u_{jk}u_{jk'}]\right|\leq \frac{2^{j(1-2\alpha \beta)}}{2^{\beta+2\alpha \beta}}&\left\{\frac{\tilde{C}(\alpha ,\beta)}{(2k'-2+c_{1,k,k'}\wedge c_{2,k,k'})^{4-2\alpha \beta}}\right.\\
             & \left.+\frac{\prod_{l=0}^{3}|2\alpha \beta-l|}{(2(k-k')-2+c_{3,k,k'})^{4-2\alpha \beta}}\right\}.
        \end{split}
  \end{equation}
    \item For $\alpha >\frac{1}{2}$. Let $j\geq 1$ and $k,k'\in \{1,...,2^j\}$ such that $k'<k$, we remark that
    \begin{equation}\label{I2 I3}
    \begin{split}
       I_2+I_3=& \frac{8\alpha ^3(2\alpha -1)\beta(\beta-1)(\beta-2)}{(2k-2+c_{1,k,k'})^{2-2\alpha }(2k'-2+c_{2,k,k'})^{2-2\alpha }}\\
    &\times\frac{1}{\left((2k-2+c_{1,k,k'})^{2\alpha }+(2k'-2+c_{2,k,k'})^{2\alpha }\right)^{2-\beta}}.
    \end{split}
    \end{equation}
    And by the inequality $\frac{ab}{a^2+b^2}\leq \frac{1}{2}$, we have
    \begin{equation}\label{I4}
        \begin{split}
       |I_4|\leq & \frac{4\alpha ^4\beta|\beta-1||\beta-2||\beta-3|}{(2k-2+c_{1,k,k'})^{2-2\alpha }(2k'-2+c_{2,k,k'})^{2-2\alpha }}\\
    &\times\frac{1}{\left((2k-2+c_{1,k,k'})^{2\alpha }+(2k'-2+c_{2,k,k'})^{2\alpha }\right)^{2-\beta}}.
    \end{split}
    \end{equation}
    Combining \eqref{derivative of Psi}, \eqref{I2 I3} and \eqref{I4}, we get
    \begin{equation}\label{Il 12}
          \sum_{l=1}^{4}|I_l|\leq \frac{\tilde{C}(\alpha ,\beta)}{(2k'-2+c_{1,k,k'}\wedge c_{2,k,k'})^{4-2\alpha \beta}},
        \end{equation}
        where  $\tilde{C}(\alpha ,\beta)$ is given by \eqref{C tilde}. According to \eqref{estimation ujkujkprime} and \eqref{Il 12}, we obtain
        \begin{equation}\label{estimation ujkujkprime H superieure 12}
        \begin{split}
           \left|\E[u_{jk}u_{jk'}]\right|\leq \frac{2^{j(1-2\alpha \beta)}}{2^{\beta+2\alpha \beta}}&\left\{\frac{\tilde{C}(\alpha ,\beta)}{(2k'-2+c_{1,k,k'}\wedge c_{2,k,k'})^{4-2\alpha \beta}}\right.\\
             & \left.+\frac{\prod_{l=0}^{3}|2\alpha \beta-l|}{(2(k-k')-2+c_{3,k,k'})^{4-2\alpha \beta}}\right\}.
        \end{split}
  \end{equation}
  \end{itemize}
  We put
  $$C=\frac{\tilde{C}(\alpha ,\beta)\vee \prod_{l=0}^{3}|2\alpha \beta-l|}{2^{\beta+2\alpha \beta}}.$$
  So for all $\alpha \in (0,1)$, $\beta\in (0,1]$, $j\geq 1$ and $k,k'\in\{1,...,2^j\}$ such that $k'<k$, we have
   \begin{equation}\label{estimation ujkujkprime for all H}
        \begin{split}
           \left|\E[u_{jk}u_{jk'}]\right|\leq C 2^{j(1-2\alpha \beta)}&\left\{\frac{1}{(2k'-2+\kappa_{k,k'})^{4-2\alpha \beta}}\right.\\
             & \left.+\frac{1}{(2(k-k')-2+c_{k,k'})^{4-2\alpha \beta}}\right\},
        \end{split}
  \end{equation}
  where $\kappa_{k,k'}= c_{1,k,k'}\wedge c_{2,k,k'}$ and $c_{k,k'}=c_{3,k,k'}$. This finishes the proof of  \eqref{estmation Eujkujkprime}.\par
  Now we will prove \eqref{estmation Eujk2}. For this end, let us start with proving the upper bound, for all $j\geq 1$ and $k\in \{1,...,2^j\}$, we have by \eqref{ujkujkprime} and the mean value theorem, there exist $c_{1,k},c_{2,k}\in (0,2)$ such that
  \begin{align}
  \E[|u_{jk}|^2] &= \frac{2^{j(1-2\alpha \beta)}}{2^{\beta+2\alpha \beta}}(\partial^2_y\partial^2_x\Psi_{k,k}(c_{1,k},c_{2,k})-\Delta^4\Phi_{k,k}(0)) \nonumber \\
                    &= \frac{2^{j(1-2\alpha \beta)}}{2^{\beta+2\alpha \beta}}(\partial^2_y\partial^2_x\Psi_{k,k}(c_{1,k},c_{2,k})+8-2^{2\alpha \beta+1}). \label{ujkujkprime derivative x y}
\end{align}
  \begin{itemize}
    \item For $\alpha \leq \frac{1}{2}$. By \eqref{derivative of Psi} we remark that $\partial^2_y\partial^2_x\Psi_{k,k}(c_{1,k},c_{2,k})\leq 0$, so by using \eqref{ujkujkprime derivative x y} we have
        \begin{equation}\label{estimation of Eujk2}
          \E[|u_{jk}|^2]\leq \frac{8-2^{2\alpha \beta+1}}{2^{\beta+2\alpha \beta}}2^{j(1-2\alpha \beta)}.
        \end{equation}
    \item For $\alpha >\frac{1}{2}$. By \eqref{derivative of Psi} we remark that $I_1,I_4\leq 0$, and by \eqref{I2 I3} we get, for all $j\geq 1$ and $k\in \{2,...,2^j\}$,
    \begin{equation}\label{I2 I3 estimation}
      I_2+I_3\leq \frac{\alpha ^3(2\alpha -1)\beta(\beta-1)(\beta-2)}{2^{3-2\alpha \beta-\beta}}.
    \end{equation}
    So  \eqref{ujkujkprime derivative x y} and \eqref{I2 I3 estimation} entail
    \begin{equation}\label{estimation of Eujk2 H sup 12}
    \begin{split}
       &\E[|u_{jk}|^2]\\
         &\leq 2^{j(1-2\alpha \beta)}\left(\frac{\alpha ^3(2\alpha -1)\beta(\beta-1)(\beta-2)}{8}+\frac{8-2^{2\alpha \beta+1}}{2^{\beta+2\alpha \beta}}\right).
    \end{split}
        \end{equation}
        Now for $j\geq 1$ and $k=1$, we have by \eqref{ujk} and \eqref{self similar},
        \begin{equation}\label{Euj12}
          \E[|u_{j1}|^2]=2^{j(1-2\alpha \beta)}\frac{4\E\left[B^{\alpha ,\beta}(1)-\tfrac{1}{2}B^{\alpha ,\beta}(2)\right]^2}{2^{2\alpha \beta}}.
        \end{equation}
          \end{itemize}
        Put
        \begin{equation}\label{mtildde 2}
        \begin{split}
           m_2=\max &\left\{\left(\frac{\alpha ^3(2\alpha -1)\beta(\beta-1)(\beta-2)}{8}+\frac{8-2^{2\alpha \beta+1}}{2^{\beta+2\alpha \beta}}\right);\right. \\
             & \qquad\left.\frac{4\E\left[B^{\alpha ,\beta}(1)-\tfrac{1}{2}B^{\alpha ,\beta}(2)\right]^2}{2^{2\alpha \beta}}\right\}.
        \end{split}
        \end{equation}
        Combining  \eqref{estimation of Eujk2 H sup 12}, \eqref{Euj12}, \eqref{mtildde 2} and \eqref{estimation of Eujk2}, we get for all $\alpha \in (0,1)$, $\beta\in (0,1]$, $j\geq 1$ and $k\in \{1,...,2^j\}$,
        \begin{equation}\label{Eujk m tilde}
          \E[|u_{jk}|^2]\leq m_2 2^{j(1-2\alpha \beta)}.
        \end{equation}
  This finishes the proof of the upper bound in \eqref{estmation Eujk2}. Let us now investigate the lower bound of \eqref{estmation Eujk2}.  According to \eqref{ujk} and \eqref{self similar}, it follows that for all $j\geq 1$ and $k\in \{1,...,2^j\}$,
    \begin{equation}\label{Eujk k small}
       \E[|u_{jk}|^2]=2^{j(1-2\alpha \beta)}\frac{4C(k)}{2^{2\alpha \beta}},
    \end{equation}
    where
    $$C(k)=\E\left[B^{\alpha ,\beta}(2k-1)-\frac{1}{2}B^{\alpha ,\beta}(2k)-\frac{1}{2}B^{\alpha ,\beta}(2k-2)\right]^2.$$
We know by \cite[Lemma 3.3]{Ait Ouahra} (In this last article to prove the local-nondeterminism of the bBm the authors use the same techniques as those of \cite{Mendy}) that the process $(B^{\alpha ,\beta}(t))_{t\geq 0}$  is locally non-deterministic i.e. for all $0=t_0<t_1<...<t_m<1$ with $t_m-t_1<\delta$ and $(u_1,...,u_m)\in\R^m$,
\begin{equation}\label{SLND BHK}
  \var\left(\sum_{j=1}^{m}u_j[B^{\alpha ,\beta}(t_{j})-B^{\alpha ,\beta}(t_{j-1})]\right)\geq C_m \sum_{j=1}^{m}u_j^2\var\left(B^{\alpha ,\beta}(t_{j})-B^{\alpha ,\beta}(t_{j-1})\right).
\end{equation}
On the other hand by \eqref{self similar}, we have for some $\varepsilon<\frac{\delta}{2}\wedge\frac{1}{2k}$,
\begin{equation}\label{Ck}
\begin{split}
   C(k) = \varepsilon^{-2\alpha \beta}\E &\left[\frac{1}{2}\left[B^{\alpha ,\beta}(\varepsilon(2k-1))-B^{\alpha ,\beta}(\varepsilon(2k-2))\right]\right. \\
     & \left.-\frac{1}{2}\left[B^{\alpha ,\beta}(\varepsilon(2k))-B^{\alpha ,\beta}(\varepsilon(2k-1))\right]\right]^2.
\end{split}
\end{equation}
Combining \eqref{SLND BHK}, \eqref{Ck} and \eqref{var increment}, we get $ C(k) \geq \frac{C_3}{2^{1+\beta}}.$ Hence
$$ \E[|u_{jk}|^2] \geq m_1 2^{j(1-2\alpha \beta)} \qquad\text{with}\qquad m_1=\frac{4C_3}{2^{2\alpha \beta+\beta+1}}.$$
This finishes the proof of Lemma \ref{lem estimation Eujkujkprime}.
\end{proof}
\begin{rem}
    We remark that when $\beta=1$, $\partial^2_y\partial^2_x\Psi_{k,k'}(c_{1,k,k'},c_{2,k,k'})=0$ and hence $\Delta^2_y\Delta^2_x\Psi_{k,k'}(0,0)=0$, so equation \eqref{ujkujkprime} becomes
\begin{equation}\label{ujkujkprime delta0}
  \E[u_{jk}u_{jk'}]=-\frac{2^{j(1-2\alpha )}}{2^{1+2\alpha }}\Delta^4\Phi_{k,k'}(0),
\end{equation}
and this is the same equation as the (IV.9) in \cite{Roynette}, for the fractional Brownian motion.
\end{rem}
\begin{lem}\label{vjkvjk rimep 111}
  There exists a constant $M>0$ such that, for all $j\geq 1$ and $k,k'\in \{1,...,2^j\}$, we have
  \begin{equation}\label{sumEujujprime}
    \sum_{k,k'=1}^{2^j}|\E v_{jk}v_{jk'}|^2\leq M 2^j.
  \end{equation}
\end{lem}
\begin{proof}
Equality \eqref{vjk} and H\"{o}lder's inequality  give
  \begin{align*}
    \sum_{k,k'=1}^{2^j}|\E v_{jk}v_{jk'}|^2 &= 2\sum_{k'<k,\,k'\geq 2 \atop k-k'\geq 2}^{2^j}|\E v_{jk}v_{jk'}|^2 +2\sum_{k'<k,\,k'\geq 2 \atop k-k'=1}^{2^j}|\E v_{jk}v_{jk'}|^2 +2\sum_{k=2}^{2^j}|\E v_{jk}v_{j1}|^2\\
                                            &\qquad +\sum_{k=1}^{2^j}\{\E [|v_{jk}|^2]\}^2\\
                                            &\leq 2\sum_{k'<k,\,k'\geq 2 \atop k-k'\geq 2}^{2^j}\left|\frac{\E u_{jk}u_{jk'}}{\sigma_{jk}\sigma_{jk'}}\right|^2 +2(2^j-2) +2(2^j-1)+2^j\\
                                            &= 2J+5\,2^j-6.
  \end{align*}
  We will estimate  $J$. For this end, let us note $A=\frac{2C^2}{m_2^2}$, so we obtain by \eqref{estmation Eujkujkprime} and \eqref{estmation Eujk2},\\
  \begin{align*}
    J &\leq A\sum_{k'<k,\,k'\geq 2 \atop k-k'\geq 2}^{2^j} \left\{\frac{1}{(2k'-2+\kappa_{k,k'})^{8-4\alpha \beta}}+\frac{1}{(2(k-k')-2+c_{k,k'})^{8-4\alpha \beta}}\right\}\\
     &= A\sum_{k=4}^{2^j}\sum_{k'=2}^{k-2} \left\{\frac{1}{(2k'-2+\kappa_{k,k'})^{8-4\alpha \beta}}+\frac{1}{(2(k-k')-2+c_{k,k'})^{8-4\alpha \beta}}\right\} \\
     &\leq  A\sum_{k=4}^{2^j}\sum_{k'=2}^{k-2} \left\{\frac{1}{(2k'-2)^{8-4\alpha \beta}}+\frac{1}{(2(k-k')-2)^{8-4\alpha \beta}}\right\} \\
     &\leq A\sum_{k=4}^{2^j}\sum_{k'=2}^{k-2} \left\{\int_{2k'-3}^{2k'-2}\frac{1}{x^{8-4\alpha \beta}}dx+\int_{2(k-k')-3}^{2(k-k')-2}\frac{1}{x^{8-4\alpha \beta}}dx\right\} \\
     &= 2A\sum_{k=4}^{2^j}\sum_{k'=2}^{k-2}\int_{2k'-3}^{2k'-2}\frac{1}{x^{8-4\alpha \beta}}dx \\
     &\leq 2A\sum_{k=4}^{2^j}\int_{1}^{2k-6}\frac{1}{x^{8-4\alpha \beta}}dx \\
     &\leq \frac{2A}{7-4\alpha \beta}(2^j-3).
  \end{align*}
  And this proves Lemma \ref{vjkvjk rimep 111}.
\end{proof}
\begin{lem}\label{vjk mmoins cp 111}
For all $j\geq1$ and $k\in \{1,...,2^j\}$, we have
  \begin{equation}\label{vjk moins cp 111}
    \E\left[\sum_{k=1}^{2^j}(|v_{jk}|^p-c_p)\right]^2\leq (c_{2p}-c_p^2)M 2^j,
  \end{equation}
  where $c_p=\frac{1}{\sqrt{2\pi}}\int_{\R}|x|^pe^{-\tfrac{x^2}{2}}dx.$
\end{lem}
\begin{proof}
  First we get
  $$\E\left[\sum_{k=1}^{2^j}(|v_{jk}|^p-c_p)\right]^2=\sum_{k,k'=1}^{2^j}\E\left[(|v_{jk}|^p-c_p)(|v_{jk'}|^p-c_p)\right].$$
  And by applying Lemma \ref{cauchy gaussian}, with $f(x)=g(x)=|x|^p-c_p$,  we obtain
  $$\E\left[\sum_{k=1}^{2^j}(|v_{jk}|^p-c_p)\right]^2\leq (c_{2p}-c_p^2)\sum_{k,k'=1}^{2^j}|\E[v_{jk}v_{jk'}]|^2.$$
  Inequality \eqref{sumEujujprime} of Lemma \ref{vjkvjk rimep 111} ends the proof of Lemma \ref{vjk mmoins cp 111}.
\end{proof}
Now, we are ready to show Theorem \ref{pricipal BHK}
\begin{proof}[of Theorem \ref{pricipal BHK}]
We are going to prove that, almost surly
\begin{equation}\label{convergence to}
  2^{-j}\sum_{k=1}^{2^j}|v_{jk}|^p\underset{j \to \infty}\longrightarrow c_p.
\end{equation}
For this end we will show that for all $\varepsilon>0$ we have
\begin{equation}\label{serie}
  \sum_{j\geq 1} \mathbb{P}\left\{2^{-j}\sum_{k=1}^{2^j}|v_{jk}|^p\notin [c_p-\varepsilon,c_p+\varepsilon]\right\}<\infty.
\end{equation}
Markov's inequality gives
\begin{equation}\label{inferieur}
  \mathbb{P}\left\{2^{-j}\sum_{k=1}^{2^j}|v_{jk}|^p\notin [c_p-\varepsilon,c_p+\varepsilon]\right\}\leq \frac{1}{\varepsilon^22^{2j}}\E\left[\sum_{k=1}^{2^j}(|v_{jk}|^p-c_p)\right]^2.
\end{equation}
Combining the inequality \eqref{inferieur} and Lemma  \ref{vjk mmoins cp 111}, we get that \eqref{serie} holds and \eqref{convergence to} is then a consequence of   Borel-Cantelli Lemma. Finally our main result, Theorem \ref{pricipal BHK}, is a simple consequence of Theorem \ref{Besov}.
\end{proof}
\begin{rem}
\begin{enumerate}
  \item Let $(Y^\alpha_t)_{t\geq0}$ be a sub-fractional Brownian motion i.e. a mean zero Gaussian process with covariance function
  $$\E[Y^\alpha_tY^\alpha_s]=s^{2\alpha}+t^{2\alpha}-\frac{1}{2}\left[(s+t)^{2\alpha}+|t-s|^{2\alpha}\right],$$
 where $\alpha \in (0,1)$. We believe that by the same calculations as in the above one can get that almost all paths of the sub-fractional Brownian motion belong (resp. do not belong) to the Besov spaces  $\mathbf{Bes}(\alpha ,p)$ (resp. $\mathbf{bes}(\alpha, p)$).
  \item Let $r(t)$ be a real valued function, such that the kernel $K_r(t,s)$ defined by
  $$K_r(t,s):=r(t)-r(s)-r(t-s),$$
  is positive on the real line, and let $\varphi$ be as follows
  $$\varphi(t)=\int_{0}^{\infty}(1-e^{-u|t|})dm(u),$$
  where $m$ is a positive measure on $[0,\infty)$ such that $\int_{1}^{\infty}dm(u)<\infty$. Therefore by \cite[Theorem 5.1.]{AlpayLevanony}, we get that
  $$K(s,t):=\varphi(r(t)+r(s))-\varphi(r(t-s)),$$
  is a positive Kernel on the real line. Let $(X_t)_{t\geq 0}$ be a centred Gaussian process with covariance function
      $$\E[X_tX_s]=K(s,t).$$
         If in addition we assume that $\varphi$ and $r$ are in $\mathcal{C}^4((0,\infty))$, and that for all $a>0$, we have $r(at)=a^{\alpha}r(t)$ and $\varphi(ax)=a^{\beta}\varphi(x)$ (i.e. $X$ is a self-similar Gaussian process with index $\frac{\alpha\beta}{2}$). With the analogous (but more involved) computations as above, one can get that almost all paths of the process $X$ belong (resp. do not belong) to the Besov spaces  $\mathbf{Bes}(\alpha\beta/2,p)$ (resp. $\mathbf{bes}(\alpha\beta/2, p)$).
\end{enumerate}
\end{rem}
\section{An It\^{o}-Nisio theorem for the bifractional Brownian motion.}
 Let $\mathcal{E}$ be the linear space generated by the  indicator functions $\mathbbm{1}_{[0,t]}$ endowed with the inner product
 \begin{equation}\label{inner product}
   <\mathbbm{1}_{[0,t]},\mathbbm{1}_{[0,s]}>_{\mathcal{H}}=R(s,t)\,,\,\,\, s, t\in[0,1].
 \end{equation}
Define $\mathcal{H}$ as the completion of $\mathcal{E}$ w.r.t. the inner product $<.,.>_{\mathcal{H}}$. The application $\varphi\in \mathcal{H} \to B^{\alpha,\beta}(\varphi)$ is an isometry from $\mathcal{H}$ to the Gaussian space generated by $B^{\alpha,\beta}$. $B^{\alpha,\beta}(\varphi)$ is the Wiener integral of $\varphi$ w.r.t. $B^{\alpha,\beta}$.

Our main result in this paragraph is the following It\^{o}-Nisio theorem for the bifractional Brownian motion
\begin{thm}\label{Ito Nisio}
    Let $\alpha \beta>\frac{1}{2}$ and $(\varphi_n)_{n\geq1}$ be an orthono rmal  basis of $\mathcal{H}$. Then we have almost surly
    $$\sum_{n=1}^{N}<\varphi_n,\mathbbm{1}_{[0,t]}>_{\mathcal{H}}B^{\alpha,\beta}(\varphi_n)\underset{N \to +\infty}\longrightarrow B^{\alpha,\beta}(t)\quad\text{in the Besov space }
\mathbf{Bes}(\alpha \beta-\varepsilon,p),$$
where $\varepsilon>0 $ and $p\geq1$ are such that
$\frac{1}{2}<\alpha \beta-\varepsilon-\frac{1}{p} $.
\end{thm}
By classical continuous injections we can deduce
\begin{cor}
 Suppose that $\alpha\beta> \frac{1}{2} $ and $(\varphi_n)_{n\geq1}$ be an orthonormal  basis of $\mathcal{H}$. Then we have almost surly
    $$\sum_{n=1}^{N}<\varphi_n,\mathbbm{1}_{[0,t]}>_{\mathcal{H}}B^{\alpha,\beta}(\varphi_n)\underset{N \to +\infty}\longrightarrow B^{\alpha,\beta}(t)\quad\text{in the H\"{o}lder space }
\mathbf{C}^{\gamma},$$
 for any $\gamma<\alpha\beta$.
\end{cor}
\begin{proof}[of Theorem \ref{Ito Nisio}]
 Put
  $\quad X_N(t)=\displaystyle\sum_{n=1}^{N}<\varphi_n,\mathbbm{1}_{[0,t]}>_{\mathcal{H}}B^{\alpha ,\beta}(\varphi_n).$ And define

  \begin{equation*}
  z_{jk}:=2\cdot 2^{j/2}\left\{X_N\left(\frac{2k-1}{2^{j+1}}\right)-\frac{1}{2}X_N\left(\frac{2k}{2^{j+1}}\right)-\frac{1}{2}X_N\left(\frac{2k-2}{2^{j+1}}\right)\right\}.
\end{equation*}
Let $\{h_{jk},\;j\geq 0,\;k=1,...,2^j\} $ be the Haar functions defined as follows
$$h_{jk}=\sqrt{2^j}\mathbbm{1}_{[\frac{2k-2}{2^{j+1}},\frac{2k-1}{2^{j+1}}[}-\sqrt{2^j}\mathbbm{1}_{[\frac{2k-1}{2^{j+1}},\frac{2k}{2^{j+1}}[}\qquad\text{and}\qquad h_1=\mathbbm{1}_{[0,1]}.$$
  Remark that
  \begin{equation*}
    \begin{split}
        u_{jk}-z_{jk}&= B^{\alpha ,\beta}(h_{jk})-X_N(h_{jk}) \\
         &= B^{\alpha ,\beta}(h_{jk})-\sum_{n=1}^{N}<\varphi_n,h_{jk}>_{\mathcal{H}}B^{\alpha ,\beta}(\varphi_n).
    \end{split}
  \end{equation*}
  Set
  \begin{equation}\label{omega}
    \omega_{jk}^{N}=\frac{u_{jk}-z_{jk}}{\varrho^N_{jk}}\qquad\text{with}\qquad\varrho^N_{jk}=\{\E[|u_{jk}-z_{jk}|^2]\}^{1/2}.
  \end{equation}
  First, by Borel-Cantelli lemma, we can easily show that almost surly
  \begin{equation}\label{cv ujk moins zjk}
    2^{-j(1+\varepsilon p)}\sum_{k=1}^{2^j}|\omega_{jk}^{N}|^p\underset{j \to \infty}\longrightarrow 0.
  \end{equation}
Let $\phi, \psi\in \mathcal{H}$ and $\sum_{i=1}^{L_n} \lambda^n_i\mathbbm{1}_{[0,t^n_i]}$,  $\sum_{j=1}^{M_m} \mu^m_j\mathbbm{1}_{[0,s^m_j]}$ be two sequences in $\mathcal{E}$ such that
$$\phi=\lim_{n\to \infty}\sum_{i=1}^{L_n} \lambda^n_i\mathbbm{1}_{[0,t^n_i]}\qquad\text{and}\qquad \psi=\lim_{m\to \infty} \sum_{j=1}^{M_m} \mu^m_j\mathbbm{1}_{[0,s^m_j]}\qquad\text{in }\mathcal{H}.$$
Define for all $a>0$,
$$\phi^a=\lim_{n\to \infty}\sum_{i=1}^{L_n} \lambda^n_i\mathbbm{1}_{[0,at^n_i]}\qquad\text{and}\qquad \psi^a=\lim_{m\to \infty} \sum_{j=1}^{M_m} \mu^m_j\mathbbm{1}_{[0,as^m_j]}\qquad\text{in }\mathcal{H}.$$
One can see easily by \eqref{self similar} and a density argument that
\begin{equation}\label{self semilar general}
  \E[B^{\alpha ,\beta}(\phi^a)B^{\alpha ,\beta}(\psi^a)]=a^{2\alpha \beta}\E[B^{\alpha ,\beta}(\phi)B^{\alpha ,\beta}(\psi)].
\end{equation}
Set $\theta^n_{j,k}=<\varphi_n,h_{jk}>_{\mathcal{H}}$, we get
   \begin{align}
    |\varrho^N_{jk}|^2 &= \E\left[B^{\alpha ,\beta}(h_{jk})-\sum_{n=1}^{N}<\varphi_n,h_{jk}>_{\mathcal{H}}B^{\alpha ,\beta}(\varphi_n)\right]^2 \nonumber \\
    &= ||h_{jk}||^2_{\mathcal{H}}+\sum_{n=1}^{N}|\theta^n_{j,k}|^2-2\sum_{n=1}^{N}|\theta^n_{j,k}|^2=\sum_{n=N+1}^{\infty}|\theta^n_{j,k}|^2.\label{rho 2}
   \end{align}
   Put
   $$\tilde{h}_{jk}=\mathbbm{1}_{[\frac{2k-2}{2^{j+1}},\frac{2k-1}{2^{j+1}}[}-\mathbbm{1}_{[\frac{2k-1}{2^{j+1}},\frac{2k}{2^{j+1}}[},\qquad\tilde{g}_{jk}=\mathbbm{1}_{[\frac{2k-2}{2^{(j+1)/2}},\frac{2k-1}{2^{(j+1)/2}}[}-\mathbbm{1}_{[\frac{2k-1}{2^{(j+1)/2}},\frac{2k}{2^{(j+1)/2}}[},$$
  \begin{equation}\label{gkk}
    g_{k}=\mathbbm{1}_{[2k-2,2k-1[}-\mathbbm{1}_{[2k-1,2k[}.
  \end{equation}
    Let $a_j=2^{-(j+1)/2}$, we have by \eqref{self semilar general}
    \begin{align}\label{theta}
      \theta^n_{j,k} &= \E[B^{\alpha ,\beta}(h_{jk})B^{\alpha ,\beta}(\varphi_n)] \nonumber\\
      &=2^{j/2}\E[B^{\alpha ,\beta}(\tilde{h}_{jk})B^{\alpha ,\beta}(\varphi_n)]\nonumber\\
      &=2^{j/2}  \E[B^{\alpha ,\beta}(\tilde{g}^{a_j}_{jk})B^{\alpha ,\beta}((\varphi^{a^{-1}_j}_n)^{a_j})] \nonumber\\
      &= 2^{-\alpha \beta}2^{j(1/2-\alpha \beta)}\E[B^{\alpha ,\beta}(\tilde{g}_{jk})B^{\alpha ,\beta}(\varphi^{a^{-1}_j}_n)].
    \end{align}
    Hence combining \eqref{rho 2} and \eqref{theta}, we get
   \begin{align}\label{rho N 2}
     |\varrho^N_{jk}|^2 &= 2^{-2\alpha \beta}2^{j(1-2\alpha \beta)}\sum_{n=N+1}^{\infty}\left\{\E[B^{\alpha ,\beta}(\tilde{g}_{jk})B^{\alpha ,\beta}(\varphi^{a^{-1}_j}_n)]\right\}^2\nonumber\\
    & \leq 2^{-2\alpha \beta}2^{j(1-2\alpha \beta)}\sup_{j,k}\sum_{n=N+1}^{\infty}\left\{\E[B^{\alpha ,\beta}(\tilde{g}_{jk})B^{\alpha ,\beta}(\varphi^{a^{-1}_j}_n)]\right\}^2.
   \end{align}
  The supremum  in the last term of the inequality is finite. 
   In fact  we remark that $(a^{\alpha \beta}_j\varphi^{a^{-1}_j}_n)_{n\geq 1}$ is an orthonormal basis of $\mathcal{H}$, therefore
   \begin{align}\label{sup N}
     \sum_{n=N+1}^{\infty}\left\{\E[B^{\alpha ,\beta}(\tilde{g}_{jk})B^{\alpha ,\beta}(\varphi^{a^{-1}_j}_n)]\right\}^2 &= a^{-2\alpha \beta}_j\sum_{n=N+1}^{\infty}<\tilde{g}_{jk},a^{\alpha \beta}_j\varphi^{a^{-1}_j}_n>^2_{\mathcal{H}}\nonumber\\
     &\leq a^{-2\alpha \beta}_j\sum_{n=1}^{\infty}<\tilde{g}_{jk},a^{\alpha \beta}_j\varphi^{a^{-1}_j}_n>^2_{\mathcal{H}}\nonumber\\
     &= a^{-2\alpha \beta}_j||\tilde{g}_{jk}||^2_{\mathcal{H}}= a^{-2\alpha \beta}_j||g^{a_j}_{k}||^2_{\mathcal{H}}=||g_{k}||^2_{\mathcal{H}}.
   \end{align}
   On the other hand, we have by \eqref{Eujk m tilde} that $\E[|u_{jk}|^2]\leq m_2 2^{j(1-2\alpha \beta)}.$ Therefore by \eqref{ujk}, \eqref{self similar}, \eqref{gkk} and \eqref{inner product} we get for all $k\geq 1$,
   \begin{equation}\label{gk}
     ||g_{k}||^2_{\mathcal{H}}=\E[2B^{\alpha ,\beta}(2k-1)-B^{\alpha ,\beta}(2k)-B^{\alpha ,\beta}(2k-2)]^2\leq m_2.
   \end{equation}
   According to \eqref{sup N} we obtain for all $N\geq 1$,
   \begin{equation}\label{AN}
     A_N:=\sup_{j,k}\sum_{n=N+1}^{\infty}\left\{\E[B^{\alpha ,\beta}(\tilde{g}_{jk})B^{\alpha ,\beta}(\varphi^{a^{-1}_j}_n)]\right\}^2<\infty.
   \end{equation}
   $(A_N)_{N\geq 1}$ is a non-increasing real valued sequence such that
   \begin{equation}\label{MAN}
   \lim_{N} A_N=0 \qquad \text{a.s..}
   \end{equation}
   We derive from \eqref{omega}, \eqref{rho N 2} and \eqref{AN} the following inequality
   \begin{equation}\label{2jk}
     2^{-j(1+\varepsilon p)}\sum_{k=1}^{2^j}|\omega_{jk}^{N}|^p\geq 2^{p\alpha \beta} 2^{-jp(\frac{1}{2}-(\alpha \beta-\varepsilon)+\frac{1}{p})}\sum_{k=1}^{2^j}\frac{|u_{jk}-z_{jk}|^p}{A^{p/2}_N}.
   \end{equation}
   We remark that the sequence $2^{-jp(\frac{1}{2}-(\alpha \beta-\varepsilon)+\frac{1}{p})}\sum_{k=1}^{2^j}|u_{jk}-z_{jk}|^p$ is increasing in $j$ for $\varepsilon$ small enough, $p$ large enough and $\alpha \beta>\frac{1}{2}$. Therefore by  \eqref{cv ujk moins zjk} and \eqref{2jk}, we get that almost surely,
   $$ \sup_{j\geq0} 2^{-jp(\frac{1}{2}-(\alpha \beta-\varepsilon)+\frac{1}{p})}\sum_{k=1}^{2^j}|u_{jk}-z_{jk}|^p  \leq 2^{-p\alpha \beta}A^{p/2}_N.$$
  Which finishes the proof of Theorem \ref{Ito Nisio} by applying \eqref{MAN} and Theorem \ref{Besov} .
\end{proof}
\begin{rem}
  When $\beta=1$ we get the It\^{o}-Nisio theorem for the fractional Brownian motion with $\alpha >\frac{1}{2}$.
\end{rem}


%
%



\end{document}